\documentclass[11pt]{amsart}
\usepackage{cases}

\theoremstyle{plain}
\newtheorem{theorem}                {Theorem}      [section]
\newtheorem*{theorem*}                {Theorem}
\newtheorem{proposition}  [theorem]  {Proposition}
\newtheorem{corollary}    [theorem]  {Corollary}
\newtheorem{lemma}        [theorem]  {Lemma}

\theoremstyle{definition}

\newtheorem{remark}       [theorem]  {Remark}
\newtheorem{definition}   [theorem]  {Definition}

\DeclareMathOperator{\trace}{trace} 

\DeclareMathOperator{\Div}{div}

\DeclareMathOperator{\cst}{constant}

\DeclareMathOperator{\sol}{Sol_3}

\numberwithin{equation}{section}

\begin{document}

\title[Biconservative surfaces]{The rigidity of biconservative surfaces in $\sol$}

\author{Dorel~Fetcu}

\thanks{}

\address{Department of Mathematics and Informatics\\
Gh. Asachi Technical University of Iasi\\
Bd. Carol I, 11A \\
700506 Iasi, Romania} \email{dorel.fetcu@academic.tuiasi.ro}



\begin{abstract} We consider biconservative surfaces in $\sol$, find their local equations, and then show that all biharmonic surfaces in this space are minimal. 
\end{abstract}

\maketitle

\section{Introduction}

Almost forty years ago, B.-Y. Chen \cite{43} introduced biharmonic submanifolds of Euclidean spaces $\mathbb{E}^{n}$, as isometric immersions with harmonic mean curvature vector field.  Then, in \cite{Ishikawa, J2} it was proved that biharmonic surfaces in $\mathbb{E}^{3}$ are minimal, a result that led to the still open Chen's Conjecture that biharmonic submanifolds of Euclidean spaces are minimal \cite{43}. 

In the same era, independently and in a more abstract way, biharmonic maps were defined by G.-Y.~Jiang \cite{Jiang} as critical points of the $L^{2}-$norm of the tension field. This type of a variational problem was suggested back in the 1964 by J.~Eells and J.~H.~Sampson in their seminal paper \cite{ES}. We note that the two definitions agree in Euclidean spaces $\mathbb{E}^n$.

Since ambient spaces with non-positive curvatures do not allow for interesting compact examples, most of the research has been done on biharmonic submanifolds of Euclidean spheres and other spaces with convenient curvature properties (for detailed accounts see \cite{FO, Chen-Ou}). 

The notion of biconservative submanifolds was derived from the theory of biharmonic submanifolds by only requiring the vanishing of the tangent part of the bitension field. Although a rather new one, this subject is already well established. If we are to illustrate only the literature on biconservative surfaces, we refer to papers like \cite{CMOP, FNO, Hasanis, LMO, MOR-Euclidean, MOR, MOP, Nistor1} to gain a satisfactory (if fairly incomplete) imagine. 

Studying the geometry of surfaces in $\sol$ seems to be the most challenging among the eight Thurston geometries. The lack of some powerful tools used in other homogeneous $3$-manifolds to describe the geometry of immersed surfaces lead to many difficulties that do not appear in the case of the remaining seven geometries (see \cite{DM}). However, in spite of these difficulties, there are important results on the existence and uniqueness of constant mean curvature (CMC) spheres \cite{DM,M,MMPR}, on totally umbilic surfaces \cite{ST}, half-space theorems for minimal surfaces \cite{DMR}, as well as Jenkins-Serrin type results for minimal \cite{N} and CMC surfaces \cite{KM}. In a local approach, one classified constant angle surfaces \cite{LM}. Also CMC biharmonic surfaces were studied in \cite{OW} and it turned out that in $\sol$ they are actually minimal. 

In our paper, we first prove that CMC biconservative surfaces in $\sol$ are minimal. Next, we show that non-CMC biconservative surfaces are constant angle surfaces and find their local equations. We conclude with the fact that all biharmonic surfaces in $\sol$ are minimal. This last result is similar to that for biharmonic surfaces in $\mathbb{E}^3$ (see \cite{43,D}). 

\textbf{Conventions.}
Throughout the paper surfaces are oriented, we will use the following sign conventions
$$
R(X,Y)Z=\nabla_{X}\nabla_{Y}Z-\nabla_{Y}\nabla_{X}Z-\nabla_{[X,Y]}Z,\quad
\Delta=\trace\nabla^{2}=\trace(\nabla\nabla-\nabla_{\nabla}),
$$
and objects on $\sol$ will be indicated by $\overline{(\cdot)}$.

\section{Preliminaries}

Biharmonic maps $\phi : M^{m} \to N^{n}$ between two Riemannian manifolds are critical points of the bienergy functional
$$
E_{2}:C^{\infty}(M,N)\to \mathbb{R}, \quad E_{2}(\phi)=\frac{1}{2}\int_{M} |\tau(\phi)|^{2} dv,
$$
where $\tau(\phi)= \trace  \nabla d \phi$ is the tension field of $\phi$. The Euler-Lagrange equation, also called the biharmonic equation, was derived by G.-Y.~Jiang \cite{Jiang}
\begin{eqnarray}\label{tau-2}
\tau_{2}(\phi)=\Delta \tau(\phi)-\trace R^N(d\phi(\cdot),\tau(\phi))d \phi (\cdot)=0,
\end{eqnarray}
where $\tau_{2}(\phi)$ is the bitension field of $\phi$.

Since any harmonic map is biharmonic, the important case, from the point of view of biharmonicity, is that of non-harmonic biharmonic ones.

Next, if we consider a fixed map $\phi$ and let the domain metric vary, one obtains a functional on the set $\mathcal{G}$ of
Riemannian metrics on $M$
$$
\mathcal{F}_{2}:\mathcal{G}\to \mathbb{R}, \quad \mathcal{F}_{2}(g)=E_{2}(\phi).
$$
Critical points of this functional are characterized by the vanishing of the stress-energy
tensor $S_{2}$ of the bienergy (see \cite{LMO}). This tensor was introduced
in \cite{Jiang87} as
\begin{eqnarray*}
S_{2}(X,Y)&=&\frac{1}{2}\vert \tau (\phi)\vert ^{2}\langle X,Y \rangle +\langle d \phi, \nabla \tau (\phi) \rangle \langle X, Y \rangle-\langle d\phi (X),\nabla_{Y} \tau (\phi)\rangle
\\
&\ & -\langle d\phi (Y),\nabla_{X} \tau (\phi)\rangle,
\end{eqnarray*}
and it satisfies
$$
\Div S_{2}=-\langle \tau_{2}(\phi), d\phi\rangle.
$$

For isometric immersions, $(\Div S_{2})^{\sharp} =-\tau_{2}(\phi)^{\top}$, where $\tau_{2}(\phi)^{\top}$ is the
tangent part of the bitension field.

\begin{definition}
A submanifold $\phi:M^{m} \to N^{n}$ of a Riemannian manifold $N^{n}$ is called biconservative if $\Div S_{2}=0$.
\end{definition}

From the definition, it is easy to see that a submanifold is biconservative if and only if the tangent part of the bitension field vanishes.

Now, let us consider a surface $\Sigma^2$ in a Riemannian manifold $N^n$, with the unit normal vector field $\xi$. The Gauss and the Weingarten equations of the surface
$$
\nabla^N_XY=\nabla_XY+\sigma(X,Y)\quad\textnormal{and}\quad \nabla^N_X\xi=-AX,
$$
hold for all vector fields $X$ and $Y$ tangent to the surface, where $\nabla$ is the induced connection on $\Sigma^2$, $\sigma$ is its second fundamental form, and $A$ its shape operator. The mean curvature vector field of $\Sigma^2$ is given by $H=f\xi$, where $f=(1/2)\trace A$ is the mean curvature function. 

If $f$ is constant, then $\Sigma^2$ is called a constant mean curvature (CMC) surface.

The Codazzi equation of the surface reads, for any vector fields $X$, $Y$, $Z$ tangent to $\Sigma^2$, and any normal vector field $V$,
\begin{equation}\label{eq:Codazzi}
\begin{array}{cl}
\langle \bar R(X,Y)Z,V\rangle=&\langle(\nabla^{\perp}_X\sigma)(Y,Z),V\rangle-\langle(\nabla^{\perp}_Y\sigma)(X,Z),V\rangle,
\end{array}
\end{equation}
where $(\nabla^{\perp}_X\sigma)(Y,Z)=\nabla^{\perp}_X\sigma(Y,Z)-\sigma(\nabla_XY,Z)-\sigma(Y,\nabla_XZ)$.

The Gauss equation of the surface is 
\begin{equation}\label{K}
\langle R(X,Y)Z,W\rangle=\langle\bar R(X,Y)Z,W\rangle+\langle AY,Z\rangle\langle AX,W\rangle-\langle AX,Z\rangle\langle AY,W\rangle,
\end{equation}
for all tangent vector fields $X$, $Y$, $Z$, and $W$, where $R$ is the curvature tensor of $\Sigma^2$.

Specializing a general result in \cite{LMO} in the case of surfaces, we have the following splitting of biharmonic equation theorem.

\begin{theorem}[\cite{LMO}]\label{decomposition}
A surface $\Sigma^2$ in a Riemannian manifold $N^n$ is biharmonic if and only if
$$
\begin{cases}
\Delta f-f|A|^2-f\langle\trace(R^N(\cdot,\xi)\cdot),\xi\rangle=0\\
A(\nabla f)+f\nabla f+f\trace(R^N(\cdot,\xi)\cdot)^{\top}=0.
\end{cases}
$$ 
\end{theorem}

\begin{corollary}\label{c:bicons}
A surface $\Sigma^2$ in a Riemannian manifold $N^n$ is biconservative if and only if
\begin{equation}\label{eq:bicons}
A(\nabla f)+f\nabla f+f\trace(R^N(\cdot,\xi)\cdot)^{\top}=0.
\end{equation}
\end{corollary}

Next, we shall briefly recall some basic facts on $\sol$. This Lie group is $\mathbb{R}^3$ with the Riemannian metric
$$
\langle,\rangle=e^{2z}dx^2+e^{-2z}dy^2+dz^2,
$$
where $(x,y,z)$ are the canonical coordinates of $\mathbb{R}^3$.

A left-invariant orthonormal frame field $\{E_1,E_2,E_3\}$ with respect to this metric, called the canonical frame, is defined by
$$
E_1=e^{-z}\frac{\partial}{\partial x},\quad E_2=e^z\frac{\partial}{\partial y},\quad E_3=\frac{\partial}{\partial z}.
$$
The Levi-Civita connection of $\sol$ is then the following
\begin{equation}\label{nabla_sol}
\begin{array}{lll}
\bar\nabla_{E_1}E_1=-E_3,& \bar\nabla_{E_1}E_2=0,& \bar\nabla_{E_1}E_3=E_1\\ \\
\bar\nabla_{E_2}E_1=0,& \bar\nabla_{E_2}E_2=E_3,& \bar\nabla_{E_2}E_3=-E_2\\ \\
\bar\nabla_{E_3}E_1=0,& \bar\nabla_{E_3}E_2=0,& \bar\nabla_{E_3}E_3=0.
\end{array}
\end{equation}

One can see that the vertical vector field $E_3$ foliates $\sol$ by vertical geodesics. Moreover, we get that the sectional curvatures of the vertical plane fields $(E_1,E_3)$ and $(E_2,E_3)$ are equal to $-1$, while that of the horizontal plane field $(E_1,E_2)$ is equal to $1$. 

The only totally geodesic surfaces in $\sol$ are the leaves of the first two canonical foliations
$$
\mathcal{F}_1\equiv\{x=\cst\}\quad\textnormal{and}\quad\mathcal{F}_2\equiv\{y=\cst\},
$$
which are isometric to the hyperbolic plane $\mathbb{H}^2$ (see \cite{ST}).

The leaves of the third canonical foliation $\mathcal{F}_3\equiv\{z=\cst\}$ are isometric to $\mathbb{R}^2$ with its usual flat metric and are minimal.

The curvature tensor $\bar R$ of $\sol$ is given by (see \cite{ST})
\begin{eqnarray}\label{eq:barR}
\bar R(X,Y)Z&=&\langle Y,Z\rangle X-\langle X,Z\rangle Y+2\langle Z,E_3\rangle(\langle X,E_3\rangle Y-\langle Y,E_3\rangle X)\\\nonumber &&+2(\langle X,Z\rangle\langle Y,E_3\rangle-\langle Y,Z\rangle\langle X,E_3\rangle)E_3.
\end{eqnarray}

\section{Biconservative surfaces in $\sol$}

As we have seen, CMC biharmonic surfaces in $\sol$ are minimal (\cite{OW}). The following result shows that this stands true when we only keep the weaker biconservative condition.

\begin{proposition}\label{p:biconscmc}
A CMC biconservative surface in $\sol$ is minimal.
\end{proposition}

\begin{proof} From Corrolary \ref{c:bicons} and the expression \eqref{eq:barR} of the curvature tensor of $\sol$, one can easily see that a CMC surface $\Sigma^2$ is biconservative if and only if either the vector field $E_3$ is either normal or tangent to the surface. 

If $E_3$ is normal, then $\Sigma^2$ is minimal and therefore we only have to study the case when $E_3$ is tangent. In this situation, we consider an orthonormal frame field $\{X_1=E_3,X_2\}$ on the surface and, since $\bar\nabla_{E_3}E_3=0$, it follows that $AX_1=0$ and $AX_2=2fX_2$. Also, we can write $X_2$ as
$$
X_2=\cos\alpha E_1+\sin\alpha E_2,
$$
and then
$$
E_1=\cos\alpha X_2+\sin\alpha\xi,\quad\quad E_2=\sin\alpha X_2-\cos\alpha\xi,
$$
along the surface, where $\xi=\sin\alpha E_1-\cos\alpha E_2$ is a unit vector field normal to $\Sigma^2$.

Next, using formulas \eqref{nabla_sol} and the Gauss and Weingarten equations of the surface, one obtains
\begin{eqnarray*}
\bar\nabla_{X_2}E_1&=&-\cos\alpha E_3\\&=&X_2(\cos\alpha)X_2+\cos\alpha\nabla_{X_2}X_2+2f\cos\alpha\xi+X_2(\sin\alpha)\xi-2f\sin\alpha X_2,
\end{eqnarray*}
and
\begin{eqnarray*}
\bar\nabla_{X_2}E_2&=&\sin\alpha E_3\\&=&X_2(\sin\alpha)X_2+\sin\alpha\nabla_{X_2}X_2+2f\sin\alpha\xi-X_2(\cos\alpha)\xi+2f\cos\alpha X_2.
\end{eqnarray*}

Taking the inner product with $X_1$ in both equations, we get 
$$
\cos\alpha\langle\nabla_{X_2}X_2,X_1\rangle=-\cos\alpha\quad\textnormal{and}\quad \sin\alpha\langle\nabla_{X_2}X_2,X_1\rangle=\sin\alpha,
$$
which leads to $\sin\alpha\cos\alpha=0$. This means that either $E_1$ or $E_2$ is normal to $\Sigma_2$ and therefore the surface is totally geodesic.
\end{proof}

\begin{remark} In contrast to our situation, any CMC surface in a $3$-dimensional space form is biconservative \cite{CMOP} and also in the other $3$-dimensional BCV-spaces (that are local models for six  of the Thurston geometries, the exceptions being $\mathbb{H}^3$ and $\sol$) non-minimal CMC biconservative surfaces do exist \cite{MOP}.
\end{remark}

We will henceforth focus on non-CMC biconservative surfaces. Let $\Sigma^2$ be a biconservative surface in $\sol$ such that $\nabla f\neq 0$ at a point $p\in\Sigma^2$. Therefore, there exists a neighborhood $V$ of $p$ with $\nabla f\neq 0$ on $V$. Since $f$ cannot vanish on $V$, there exists on open subset $U$ of $V$ such that $f\neq 0$ at each point of $U$. Moreover, we can assume that $f>0$ on $U$. Next, we consider an orthonormal frame field $\{X_1,X_2\}$ on $U$, where
$$
X_1=\frac{\nabla f}{|\nabla f|},\quad X_2\perp X_1,\quad |X_2|=1.
$$
We now have
$$
\nabla f=X_1(f)X_1+X_2(f)X_2,
$$
which leads to
\begin{equation}\label{eq:nablaf}
X_1(f)=|\nabla f|>0\quad\textnormal{and}\quad X_2(f)=0.
\end{equation}

Next, let $H=f\xi$ be the mean curvature vector field of $\Sigma^2$, where $\xi$ is a unit vector field normal to the surface.  By a straightforward computation, using isothermal coordinates $(x,y)$ on the surface and since $\Sigma^2$ is biconservative, one obtains (see also \cite{LO}),
$$
E=\left(\nabla_{\frac{\partial}{\partial x}}A_H\right)\frac{\partial}{\partial y}-\left(\nabla_{\frac{\partial}{\partial y}}A_H\right)\frac{\partial}{\partial x}=3f\left(-\frac{\partial f}{\partial y}\frac{\partial}{\partial x}+\frac{\partial f}{\partial x}\frac{\partial}{\partial y}\right),
$$
which shows that $X_1\perp E$ on $U$, i.e., $E$ and $X_2$ are collinear. 

On the other hand, from the Codazzi equation \eqref{eq:Codazzi} and the expression of the curvature tensor \eqref{eq:barR}, we get
$$
E=\left(\nabla_{\frac{\partial}{\partial x}}A_H\right)\frac{\partial}{\partial y}-\left(\nabla_{\frac{\partial}{\partial y}}A_H\right)\frac{\partial}{\partial x}=2f\langle\xi,E_3\rangle\left(\left\langle\frac{\partial}{\partial y},E_3\right\rangle\frac{\partial}{\partial x}-\left\langle\frac{\partial}{\partial x},E_3\right\rangle\frac{\partial}{\partial y}\right),
$$
and this readily shows that $E$ is orthogonal to $E_3$. Therefore, the two vector fields $X_2$ and $E_3$ are orthogonal and, along $U$, we can write
\begin{equation}\label{eq:E}
\begin{cases}
E_1=\sin\beta\sin\theta X_1+\cos\beta X_2-\sin\beta\cos\theta\xi\\
E_2=-\cos\beta\sin\theta X_1+\sin\beta X_2+\cos\beta\cos\theta\xi\\
E_3=\cos\theta X_1+\sin\theta\xi,
\end{cases}
\end{equation}
or, equivalently,
\begin{equation}\label{eq:X}
\begin{cases}
X_1=\sin\beta\sin\theta E_1-\cos\beta\sin\theta E_2+\cos\theta E_3\\
X_2=\cos\beta E_1+\sin\beta E_2\\
\xi=-\sin\beta\cos\theta E_1+\cos\beta\cos\theta E_2+\sin\theta E_3,
\end{cases}
\end{equation}
where $\beta$ and $\theta$ are real valued functions on $U$.

From Corrolary \ref{c:bicons} and formula \eqref{eq:barR} we easily get that the biconservative condition is equivalent to
\begin{equation}\label{eq:l1}
AX_1=-\left(f+\frac{f}{|\nabla f|}\sin(2\theta)\right)X_1=\lambda_1X_1,
\end{equation}
and therefore $X_1$ is a principal direction with the corresponding eigenfunction  $\lambda_1$. It follows that also $X_2$ is a principal direction and 
\begin{equation}\label{eq:l2}
AX_2=\left(3f+\frac{f}{|\nabla f|}\sin(2\theta)\right)X_2=\lambda_2X_2.
\end{equation}

In the following, working on $U$, we will prove a sequence of three lemmas that will eventually lead to our first main result.

\begin{lemma}\label{lemma}
On the set $U$ the following identities hold:
\begin{enumerate}

\item $X_1(\theta)=-\lambda_1+\cos(2\beta)\sin\theta$;

\item $X_2(\theta)=-\sin(2\beta)$;

\item $\cos\theta\langle\nabla_{X_1}X_1,X_2\rangle=\sin(2\beta)\sin\theta$;

\item $\cos\theta\langle\nabla_{X_2}X_1,X_2\rangle=\lambda_2\sin\theta+\cos(2\beta)$;

\item $(X_1(\beta)-\langle\nabla_{X_1}X_1,X_2\rangle\sin\theta)\sin\beta=0$;

\item $X_1(\beta)\sin\beta\cos\theta=2\sin^2\beta\cos\beta\sin^2\theta$;

\item $X_2(\beta)\cos\theta=\lambda_2+\cos(2\beta)\sin\theta$;

\item $X_2(\beta)\sin\theta-\langle\nabla_{X_2}X_1,X_2\rangle=-\cos(2\beta)\cos\theta$.

\end{enumerate}
\end{lemma}

\begin{proof} From \eqref{nabla_sol}, \eqref{eq:E}, and \eqref{eq:X}, we have
\begin{eqnarray*}
\bar\nabla_{X_1}E_3&=&\sin\beta\sin\theta E_1+\cos\beta\sin\theta E_2\\&=&-X_1(\theta)\sin\theta X_1+\cos\theta\nabla_{X_1}X_1+\cos\theta\sigma(X_1,X_1)\\&&+X_1(\theta)\cos\theta\xi-\sin\theta\lambda_1X_1,
\end{eqnarray*}
and
\begin{eqnarray*}
\bar\nabla_{X_2}E_3&=&\cos\beta E_1-\sin\beta E_2\\&=&-X_2(\theta)\sin\theta X_1+\cos\theta\nabla_{X_2}X_1+X_2(\theta)\cos\theta\xi-\sin\theta\lambda_2X_2.
\end{eqnarray*}
Considering the inner product with $X_1$, $X_2$, and $\xi$, one obtains the first four equations.

Next, again using \eqref{nabla_sol}, \eqref{eq:E}, and \eqref{eq:X}, it follows
\begin{eqnarray*}
\bar\nabla_{X_1}E_1&=&-\sin\beta\sin\theta E_3\\&=&X_1(\theta)\sin\beta\cos\theta X_1+X_1(\beta)\cos\beta\sin\theta X_1+\sin\beta\sin\theta\nabla_{X_1}X_1\\&&+\sin\beta\sin\theta\sigma(X_1,X_1)-X_1(\beta)\sin\beta X_2+\cos\beta\nabla_{X_1}X_2\\&&+X_1(\theta)\sin\beta\sin\theta\xi-X_1(\beta)\cos\beta\cos\theta\xi+\sin\beta\cos\theta\lambda_1X_1.
\end{eqnarray*}
We take the inner product with $X_2$ and obtain the fifth identity. From this and the third identity, the sixth one follows immediately.

Finally, we have
\begin{eqnarray*}
\bar\nabla_{X_2}E_1&=&-\cos\beta E_3\\&=&X_2(\theta)\sin\beta\cos\theta X_1+X_2(\beta)\cos\beta\sin\theta X_1+\sin\beta\sin\theta\nabla_{X_2}X_1\\&&+\cos\beta\sigma(X_2,X_2)-X_2(\beta)\sin\beta X_2+\cos\beta\nabla_{X_2}X_2\\&&+X_2(\theta)\sin\beta\sin\theta\xi-X_2(\beta)\cos\beta\cos\theta\xi+\sin\beta\cos\theta\lambda_2X_2,
\end{eqnarray*}
and, taking the inner product with each $X_1$, $X_2$, and $\xi$, also using the third and the fourth identities, we prove the last two items of the lemma.
\end{proof}

\begin{remark} Any two of the third, fifth, and sixth items of Lemma \ref{lemma} imply the other one. This also happens with the fourth, seventh, and eighth items. Therefore only six of the identities are independent. It is easy to verify that this is the maximum number of independent equations of this type that can be derived from the expressions of $\bar\nabla_{X_i}E_j$, $i,j\in\{1,2,3\}$.
\end{remark}

The next result shows that one of the two vector fields $E_1$ or $E_2$ is tangent to the surface $\Sigma^2$ at any point of $U$.

\begin{lemma}\label{beta} On the set $U$ we have $\sin(2\beta)=0$.
\end{lemma}

\begin{proof} From the Codazzi equation \eqref{eq:Codazzi} and \eqref{eq:barR}, one obtains
$$
\left(\nabla_{X_1}A\right)X_2-\left(\nabla_{X_2}A\right)X_1=-\sin(2\theta)X_2,
$$
which can be written as
\begin{equation}\label{codazzi_split_initial}
\begin{cases}
X_1(\lambda_2)+(\lambda_2-\lambda_1)\langle\nabla_{X_2}X_1,X_2\rangle+\sin(2\theta)=0\\
X_2(\lambda_1)-(\lambda_2-\lambda_1)\langle\nabla_{X_1}X_2,X_1\rangle=0.
\end{cases}
\end{equation}
Using Lemma \ref{lemma}, the second equation become
\begin{equation}\label{codazzi_split}
X_2(\lambda_1)\cos\theta+(\lambda_2-\lambda_1)\sin(2\beta)\sin\theta=0.
\end{equation}

In the following, we will compute $X_2(\lambda_1)$. Since $X_2(f)=0$ and $X_1(f)=|\nabla f|$, we have, also using Lemma \ref{lemma},
\begin{eqnarray}\label{eq:x2f}
X_2(\lambda_1)&=&-X_2\left(f+\frac{f}{|\nabla f|}\sin(2\theta)\right)\\\nonumber &=&\frac{f}{|\nabla f|}\left(\frac{X_2(X_1(f))}{|\nabla f|}\sin(2\theta)+2\sin(2\beta)\cos(2\theta)\right).
\end{eqnarray}

Now, since $X_2(X_1(f))=X_1(X_2(f))-[X_1,X_2](f)=-[X_1,X_2](f)$ and, from Lemma \ref{lemma}, we know that
\begin{eqnarray*}
\cos\theta[X_1,X_2]&=&\cos\theta(\nabla_{X_1}X_2-\nabla_{X_2}X_1)\\&=&-\sin(2\beta)\sin\theta X_1-(\lambda_2\sin\theta+\cos(2\beta))X_2,
\end{eqnarray*}
one obtains
\begin{equation}\label{x2x1}
\cos\theta X_2(X_1(f))=\sin(2\beta)\sin\theta X_1(f)=\sin(2\beta)\sin\theta|\nabla f|.
\end{equation}
Replacing in \eqref{eq:x2f} we get
$$
X_2(\lambda_1)=\frac{2f}{|\nabla f|}\sin(2\beta)\cos^2\theta,
$$
and then, from \eqref{codazzi_split}, since 
$$
\lambda_2-\lambda_1=4f+\frac{2f}{|\nabla f|}\sin(2\theta),
$$
one obtains
$$
\frac{2\sin(2\beta)}{|\nabla f|}\left(2|\nabla f|\sin\theta+\cos\theta\left(1+\sin^2\theta\right)\right)=0.
$$
Therefore, at each point of $U$, either 
$$
\sin(2\beta)=0\quad\textnormal{or}\quad 2|\nabla f|\sin\theta+\cos\theta\left(1+\sin^2\theta\right)=0.
$$ 

Assume that $\sin(2\beta)\neq 0$ at a point $q\in U$. Then, there exists a neighborhood $W\subset U$ of $q$ such that $\sin(2\beta)\neq 0$ at all points of $W$, which means that 
\begin{equation}\label{no}
2|\nabla f|\sin\theta+\cos\theta\left(1+\sin^2\theta\right)=0
\end{equation}
on $W$. From here, it follows that
$$
2X_2(X_1(f))\sin\theta+(2|\nabla f|\cos\theta-\sin\theta-\sin^3\theta+2\sin\theta\cos^2\theta)X_2(\theta)=0,
$$ 
on $W$. We then multiply by $\cos\theta$, use \eqref{x2x1}, Lemma \ref{lemma}, and again \eqref{no}, to obtain, after a straightforward computation, that $\cos\theta=0$ on the set $W$. But, from \eqref{no}, this implies that also $\nabla f=0$ on $W$, which is a contradiction and we conclude.
\end{proof}

\begin{lemma}\label{theta} Throughout the set $U$ we have 
$$
\cos\theta\neq 0,\quad \sin\theta\neq0, \quad X_1(\theta)=-2f,\quad X_2(\theta)=0,\quad\nabla_{X_1}X_1=0,\quad X_2(X_1(f))=0,
$$ 
and either
\begin{enumerate}
\item $\nabla_{X_2}X_1=\cos\theta X_2$;

\item $\lambda_2=-\sin\theta$,
\end{enumerate}
or
\begin{enumerate}
\item $\nabla_{X_2}X_1=-\cos\theta X_2$;

\item $\lambda_2=\sin\theta$,
\end{enumerate}
as $\beta=0$ or $\beta=\pi/2$ on $U$.
\end{lemma}

\begin{proof} The claims on $\nabla_{X_2}X_1$, $X_2(\theta)$, and $\lambda_2$ follow directly from Lemmas \ref{lemma} and~\ref{beta}. Then, using the expression of $\lambda_2$ and again the two previous lemmas, one obtains that $X_1(\theta)=-2f$.  

If $\sin\theta=0$ at a point $q\in U$, then, at this point, $\lambda_2=0$ and $\lambda_1=2f$. But, from the biconservative condition \eqref{eq:l1}, it follows that $\lambda_1=-f$ which implies that $f=0$, at $q$. Therefore, $\sin\theta\neq 0$ at any poiny of $U$.

Next, from the first equation of \eqref{codazzi_split_initial}, one can see that, if we assume $\cos\theta=0$ at a point $q\in U$, then, at that point, we have $X_1(\lambda_2)=0$. But, in general, on $U$,
$$
X_1(\lambda_2)=3X_1(f)+X_1\left(\frac{f}{|\nabla f|}\right)\sin(2\theta)+\frac{2f}{|\nabla f|}X_1(\theta)\cos(2\theta),
$$
and therefore, at $q$, 
$$
X_1(\lambda_2)=3|\nabla f|+\frac{4f^2}{|\nabla f|},
$$
as $X_1(\theta)=-2f$. Since $f>0$ and $|\nabla f|\ 0$ on $U$, this is a contradiction, i.e., $\cos\theta\neq 0$ throughout $U$. To conclude, we note that $\nabla_{X_1}X_1=0$ now follows immediately from Lemma \ref{lemma}, and $X_2(X_1(f))=0$ from equation \eqref{x2x1}.
\end{proof}

Now, we can state the rigidity of biconservative surfaces result.

\begin{theorem}\label{main1} Let $\Sigma^2$ be a biconservative surface in $\sol$ with $f>0$ and $\nabla f\neq 0$ at any point. Then, locally, $\Sigma^2$ can be parametrized as either
\begin{equation}\label{t1}
x_1(u,v)=ve^{\Psi(u)}E_1+\Phi_1(u)e^{-\Psi(u)}E_2+\Psi(u)E_3=(v,\Phi_1(u),\Psi(u)),
\end{equation}
or
\begin{equation}\label{t2}
x_2(u,v)=\Phi_2(u)e^{-\Psi(u)}E_1+ve^{-\Psi(u)}E_2+\Psi(u)E_3=(\Phi_2(u),v,\Psi(u)),
\end{equation}
where
$$
\Psi(u)=\int_{u_0}^u\cos\theta(s)ds, \Phi_1(u)=-\int_{u_0}^u\sin\theta(s)e^{\Psi(s)}ds,\Phi_2(u)=\int_{u_0}^u\sin\theta(s)e^{-\Psi(s)}ds.
$$
The function $\theta=\theta(u)$ is given either by
\begin{equation}\label{thetanice}
\theta(u)=2\arctan e^{(1-\sqrt{13})u/3},\quad u<0,
\end{equation}
or by the following first order ODE
\begin{equation}\label{thetanotnice}
(\theta'+2a_1\sin\theta)^{6a_2}=c(\theta'+2a_2\sin\theta)^{6a_1},
\end{equation}
such that $\theta'<0$ and $\theta''<0$, where $a_{1,2}=(-1\pm\sqrt{13})/6$ and $c>0$ is a positive constant.
\end{theorem}

\begin{proof} From Lemma \ref{beta}, we know that either $E_1$ or $E_2$ is tangent to our surface and therefore $\Sigma^2$ is a constant angle surface as defined in \cite{LM}, where equations \eqref{t1} and \eqref{t2} of such surfaces were found. 

We will consider only the case of the first surface, that with our notations correspond to $\beta=0$, meaning that $X_2=E_1$. The computations in the second case are similar and therefore we will omit them.

In \cite{LM}, the local coordinates $(u,v)$ were chosen such that 
$$
X_1=\frac{\partial}{\partial u}\quad\textnormal{and}\quad X_2=\alpha(u,v)\frac{\partial}{\partial v}, 
$$
for some function $\alpha$, which is allowed by the fact that
$$
[X_1,X_2]=-\nabla_{X_1}X_2=-\cos\theta X_2.
$$
Therefore, in this setting, $X_2(f)=0$ and $X_2(\theta)=0$ imply that $f=f(u)$ and $\theta=\theta(u)$. It is now straightforward to verify that surfaces defined by \eqref{t1} and \eqref{t2} satisfy all identities in Lemma \ref{theta}.

Thus, we only have to find the function $\theta=\theta(u)$ by solving the biconservative equation, which, for the first surface, taking into account Lemma \ref{theta}, can be written as
$$
\lambda_1=-\left(f+\frac{f}{f'}\sin(2\theta)\right)=2f+\sin\theta,
$$
or, equivalently,
\begin{equation}\label{bi1}
3ff'+f'\sin\theta+f\sin(2\theta)=0,
\end{equation}
where $\theta'=-2f$. Denote $y=\sin\theta$ and then we can see that
$$
2yy'=2\theta'\sin\theta\cos\theta=-2f\sin(2\theta),
$$
and equation \eqref{bi1} becomes
$$
3ff'+yf'-yy'=0.
$$
We can write $f'=y'(df/dy)$ and, replacing in the last equation, we have
$$
3fy'\frac{df}{dy}+yy'\frac{df}{dy}-yy'=0.
$$
Since from Lemma \ref{theta} we know that $\cos\theta\neq 0$ and we also have $\theta'=-2f\neq 0$, it follows that $y'=\theta'\cos\theta\neq 0$, which shows that the above equation is equivalent to
$$
3f\frac{df}{dy}+y\frac{df}{dy}-y=0.
$$
We denote $f=yg$ and, after a straightforward computation, this equation reads as
\begin{equation}\label{g}
y(3g+1)\frac{dg}{dy}+3g^2+g-1=0.
\end{equation}

We now have two cases as $3g^2+g-1=0$ or $3g^2+g-1\neq 0$. In the first one, $g=a_{1,2}=(-1\pm\sqrt{13})/6$ is a constant and $f=a\sin\theta$, where $a$ stands either for $a_1$ or $a_2$. Since $\theta'=-2f$, it follows that
$$
\theta(u)=2\arctan e^{-2au},\quad f(u)=\frac{2ae^{-2au}}{1+e^{-4au}},\quad\textnormal{and}\quad f'(u)=\frac{4a^2e^{-2au}(e^{-4au}-1)}{(1+e^{-4au})^2} .
$$
Since $f>0$, we see that $a=(\sqrt{13}-1)/6$, and, from $f'(u)=X_1(f)=|\nabla f|>0$, it follows that $u$ must be negative.

In the second case, a simple computation shows that equation \eqref{g} has the implicit solution
\begin{equation}\label{gfinal}
(y(g-a_1))^{6a_2}=c(y(g-a_2))^{6a_1},
\end{equation}
where $c\in\mathbb{R}$ is a positive constant. Hence, we have 
$$
(f-a_1\sin\theta)^{6a_2}=c(f-a_2\sin\theta)^{6a_1},
$$
which is just equation \eqref{thetanotnice}.
\end{proof}

\begin{remark} In the case when $\theta=\theta(u)$ is given by \eqref{thetanice}, one can compute the explicit expression of the function $\Psi=\Psi(u)$ as 
$$
\Psi(u)=\int_{u_0}^u\cos\theta(s)ds=\frac{1}{2a}\ln\left(e^{-4au}+1\right)+u+c_0,
$$
with $a=(-1+\sqrt{13})/6$, $c_0=\cst\in\mathbb{R}$.

In this case, we can also compute, using equation \eqref{K}, the Gaussian curvature of the surface
$$
K=-\cos^2\theta-2f\sin\theta=-\frac{4ae^{-4au}+(1-e^{-4au})^2}{(1+e^{-4au})^2}<0.
$$
\end{remark}

\begin{remark}\label{onlysol} The only possible solution of the biconservative equation \eqref{bi1} of the form  $f=a\sin\theta$, where $a\in\mathbb{R}$ is a constant, is the one with $a=(-1+\sqrt{13})/6$.  Otherwise, equation \eqref{thetanotnice} implies that $\theta$ is a constant and therefore $f=0$.
\end{remark}

\begin{theorem} Any biharmonic surface in $\sol$ is minimal.
\end{theorem}

\begin{proof} The case of CMC biharmonic surfaces was studied in \cite{LO} so we only have to show that if $\Sigma^2$ is a biharmonic surface in $\sol$ such that there exists a point $p\in\Sigma^2$ at which $\nabla f\neq 0$, then the surface is minimal.

As before, we consider an open neighborhood $U\subset\Sigma^2$ of the point $p$ such that $\nabla f\neq 0$ and $f>0$ on $U$. As in the proof of Theorem \ref{main1}, suffices treating only the case of a surface given by \eqref{t1}, the other one being similar. We will also use the same setting and notations as in the proof of Theorem \ref{main1}.

From Theorem \ref{decomposition}, we know that the normal part of the biharmonic equation is
$$
\Delta f-f|A|^2-f\langle\trace(\bar R(\cdot,\xi)\cdot),\xi\rangle=0.
$$
A direct computation, using Lemma \ref{theta}, shows that
$$
|A|^2=4f^2+4f\sin\theta+2\sin^2\theta.
$$
From formula \eqref{eq:barR} we find
$$
\langle\trace(\bar R(\cdot,\xi)\cdot),\xi\rangle=2\sin^2\theta.
$$
It follows that the biconservative surface $\Sigma^2$ given by \eqref{t1} is biharmonic if and only if
\begin{equation}\label{bihbicons}
\Delta f=4f(f^2+f\sin\theta+\sin^2\theta)>0.
\end{equation}

Now, let us assume that $\theta$ is given by \eqref{thetanice} and, therefore, 
$$
f(u)=\frac{2ae^{-2au}}{1+e^{-4au}},\quad a=\frac{-1+\sqrt{13}}{6}.
$$ 
From Lemma \ref{theta}, we have that
$$
\Delta f=f''+\cos\theta f',
$$
and then, a direct computation leads to 
$$
\Delta f=\frac{4e^{-6au}\left(e^{-4au}(2a^3-a^2)+e^{4au}(2a^3-a^2)+2a^2-12a^3\right)}{\left(1+e^{-4au}\right)^3}.
$$
It is easy to verify that $\Delta f<0$, in this case, which is a contradiction.

Next, let the function $\theta$ be a solution of equation \eqref{thetanotnice} and again denote $g=f/\sin\theta$, as in the proof of Theorem \ref{main1}. From Remark \ref{onlysol}, we have that $g\neq 0$ throughout $U$. 

The first derivative of $f$ is given by the biconservative equation as
$$
f'=-\frac{f\sin(2\theta)}{3f+\sin\theta},
$$ 
where we have also used $\theta'=-2f$. Taking into account Remark \ref{onlysol} (and possibly restricting to an open subset $V\subset U$) we can see that $f'$ is well defined. Using this formula, and again $\theta'=-2f$, it is straightforward to compute $f''$ and then
\begin{eqnarray*}
\Delta f&=&\frac{f}{(3f+\sin\theta)^3}\Big(36f^3(1-2\sin^2\theta)-f^2(6\sin\theta+18\sin^3\theta)\\&&-f(12\sin^2\theta-8\sin^4\theta)+2\sin^3\theta-2\sin^5\theta\Big).
\end{eqnarray*}
Replacing in \eqref{bihbicons}, and taking into account that $\sin\theta\neq 0$ on $U$, one obtains the following equation
\begin{equation}\label{pol}
\sin^2\theta P_1(g)-P_2(g)=0,
\end{equation}
where $P_1$ and $P_2$ are two polynomial functions of $u$, given by
$$
P_1(g)=100g^5+216g^4+324g^3+166g^2+32g+6
$$
and 
$$
P_2(g)=36g^3-6g^2-12g+2.
$$

Since $g$ satisfies equation \eqref{gfinal}, we can write
$$
\sin^2\theta=\frac{(g-a_2)^{b_1}}{c(g-a_1)^{b_2}},
$$
where $b_{1,2}=(6a_{1,2})/\sqrt{13}$ and $c\in\mathbb{R}$ is a positive constant. Then, equation \eqref{pol} can be written as
$$
(g-a_2)^{b_1}P_1(g)=c(g-a_1)^{b_2}P_2(g).
$$
Differentiating with respect to $u$ and using that $g'(u)\neq 0$ at any point of $U$, we obtain
$$
(g-a_2)^{b_1}(g-a_1)\left(b_1P_1+(g-a_2)\frac{dP_1}{dg}\right)=c(g-a_1)^{b_1}(g-a_2)\left(b_2P_2+(g-a_1)\frac{dP_2}{dg}\right)
$$

The last two equations imply 
$$
2(3g+1)P_1P_2+(3g^2+g-1)\left(P_1\frac{dP_2}{dg}-P_2\frac{dP_1}{dg}\right)=0.
$$
After a long but simple computation, this polynomial equation reads as
$$
25128g^8+92760g^7+85632g^6+15840g^5-19352g^4-13224g^3-1872g^2+656g+160=0,
$$
and shows that $g$ is a constant, which is a contradiction. We conclude that there are no biharmonic surfaces in $\sol$ other than the minimal ones.
\end{proof}

\begin{remark} This last result is similar to that in \cite{Ishikawa,D,J2} on biharmonic surfaces of the Euclidean space $\mathbb{E}^3$ and even to the one in \cite{CMO} about biharmonic surfaces in the Euclidean sphere $\mathbb{S}^3$. Although in this last case non-minimal biharmonic surfaces do exist, they are also CMC, and, therefore, there are no non-CMC such surfaces in $\mathbb{S}^3$.
\end{remark}

\end{document}